\documentclass[11pt,leqno,english,]{smfart}

\usepackage{eucal,dsfont,mathptmx,newtxtext, MnSymbol}

\usepackage[dvipsnames]{xcolor} 
\usepackage{xparse}
\usepackage{xr-hyper}
\usepackage[linktocpage=true,colorlinks=true,hyperindex,citecolor=Brown,linkcolor=blue]{hyperref}

\usepackage[text={6in,9.18in},centering]{geometry}
\newtheorem{theorem}{Theorem}[section]
\newtheorem{proposition}[theorem]{Proposition}
\newtheorem{lemma}[theorem]{Lemma}
\newtheorem{corollary}[theorem]{Corollary}
\theoremstyle{definition}

\theoremstyle{remark}
\newtheorem{remark}[theorem]{Remark}

\newcommand{\inv}{^{\raisebox{.2ex}{$\scriptscriptstyle-1$}}}   

\renewcommand{\labelenumi}{\upshape(\arabic{enumi})}

\newcommand{\spec}{\mathrm{Spec}(S)}

\DeclareMathOperator{\h}{\mathcal{H}}

\newcommand{\ma}{\mathfrak{a}}
\newcommand{\mb}{\mathfrak{b}}
\newcommand{\mc}{\mathfrak{c}}
\newcommand{\p}{\mathfrak{p}}
\newcommand{\ms}{\mathfrak{s}}
\newcommand{\mt}{\mathfrak{t}}

\newcommand{\m}{\mathfrak{m}}
\newcommand{\mi}{\mathfrak{i}}

\newcommand{\mj}{\mathfrak{j}}

\DeclareMathOperator{\vs}{\mathcal{V}(\textit{S})}

\DeclareMathOperator{\cz}{\textit{c}_{\!\scriptscriptstyle G}}
\DeclareMathOperator{\id}{\mathcal{I}(\textit{S})}
\DeclareMathOperator{\ids}{\mathcal{I}_{\scriptstyle s}(\textit{S})}
\DeclareMathOperator{\cs}{\mathcal{C}_{\scriptstyle s}(\textit{S})}
\DeclareMathOperator{\idss}{\mathcal{I}_{\scriptstyle s}(\textit{T})}

\begin{document}
	
\author{Amartya Goswami}

\address{
[1] Department of Mathematics and Applied Mathematics, University of Johannesburg, South Africa; [2] National Institute for Theoretical and Computational Sciences (NITheCS), South Africa.}

\email{agoswami@uj.ac.za}

\title{On semisubtractive ideals of semirings}

\date{}


% Semirings


\begin{abstract}
Our aim in this paper is to explore semisubtractive ideals of semirings. We prove that they form a complete modular lattice. We introduce Golan closures and prove some of their basic properties. We explore the relations between $Q$-ideals and semisubtractive ideals of semirings, and also study them in $s$-local semirings. We introduce two subclasses of semisubtractive ideals: $s$-strongly irreducible and $s$-irreducible, and provide various representation theorems. By endowing a topology on the set of semisubtractive ideals, we prove that the space is $T_0$, sober, connected, and quasi-compact. We also briefly study continuous maps between semisubtractive spaces. We construct $s$-congruences and prove a bijection between these congruences and semisubtractive ideals.
\end{abstract}
\maketitle

\section{Introduction}
Among others, the lack of subtractivity in semirings leads to three  types of ideals, namely, strongly subtractive, subtractive, and semisubtractive. Since all these types of ideals coincide with ideals in rings, they are primarily of interest in semirings. As shown in \cite{Gol99}, the following strict inclusions hold between these  ideals:
\begin{center}
strongly subtractive $\subset$ subtractive $\subset$ semisubtractive.
\end{center}
Among these, subtractive ideals, introduced as $k$-ideals in \cite{Hen58}, have received the most attention (see \cite{LaT65,Lat67,Mos70,Sto72,DS75,Ols78,Dal79,SA92, SA93,AA94,AA194, Gol99,AA08,Han15,Han21, JRT22, DG24}). From the above inclusions and extensive studies, we may ask how much of the theory of subtractive ideals can be generalized to semisubtractive ideals. It is worth noting that, beyond being a genuine extension of subtractive ideals, semisubtractive ideals are of interest in their own right due to their distinguished properties. For instance, prime ideals (and hence maximal ideals) of semirings are semisubtractive, and semisubtractive ideals are closed under products and sums.

Not only do semisubtractive ideals encompass a host of natural examples while generalizing many important classes of ideals in semirings, but the involvement of additive inverses in their definition also makes the theory closely resemble the theory of ideals in rings. The goal of this paper is to convince this fact to the reader.

We briefly describe the content of the paper. After declaring the underlying assumptions on semirings and recalling the bare minimum definitions in \textsection \ref{kid}, we study the closure properties of semisubtractive ideals under various operations and provide several examples in 
\textsection\ref{basic}. We obtain a significant result showing that the set of semisubtractive ideals forms a complete modular lattice. We introduce Golan closures and establish some basic properties. The purpose of
\textsection\ref{qiss} is to explore the relations between $Q$-ideals and semisubtractive ideals of semirings, whereas,
\textsection\ref{sls} is on studying semisubtractive ideals in $s$-local semirings. 
In
\textsection\ref{siid}, we introduce two subclasses of semisubtractive ideals, namely $s$-strongly irreducible and $s$-irreducible and provide various representation theorems. We characterize arithmetic semirings in terms of $s$-strongly irreducible ideals. The aim of 
\textsection\ref{sssp} is to endow a topology on the set of semisubtractive ideals and prove that the space is $T_0$, sober, connected, and quasi-compact. We also briefly study continuous maps between semisubtractive spaces. Motivated by the concept of $k$-congruences for $k$-ideals, in
\textsection\ref{scon} we construct $s$-congruences and prove a bijection between these congruences and semisubtractive ideals. We conclude the paper with remarks on further study of semisubtractive ideals.

\section{Semisubtractive ideals}\label{kid} 
By a semiring $S$, in this paper, we shall mean a commutative semiring with multiplicative identity element $1$. We shall also assume that $0\neq 1$ in our semirings.
A semiring homomorphism  preserves both the additive and multiplicative monoid structures of $S$.
An \emph{ideal} $\ma $ of $S $ is an additive submonoid of $S $ such that 
$sx\in \ma$,
for all $x\in \ma $ and $s\in S  .$ By $\id$, we shall denote the set of all ideals of $S$. We shall also use the symbol $0$ to denote the zero ideal of $S $. A proper ideal $\p$ of $S$ is called \emph{prime} if for every $\ma$, $\mb\in \id$ and $\ma\mb\subseteq \p$ implies $\ma\subseteq \p$ or $\mb\subseteq \p$. We denote the set of all prime ideals of $S$ by $\spec$.
An element $y$ of $S$ is called an \emph{additive inverse of} an element $x$ if $x+y=0$.  We shall denote the set of all elements of $S$ having additive inverses
by $\vs$. Obviously, when additive inverse of an element $x$ exists, it is unique, and we shall denote it by $-x$. 
In the next lemma, we gather some elementary properties of $\vs$.

\begin{lemma}\label{epvs}
Suppose $S$ is a semiring. Then the following hold.
\begin{enumerate}
\item\label{vsne} $\vs$ is always nonempty.

\item $\vs$ is a submonoid of $(S,+)$.

\item If $s+s'\in \vs$, then $s\in \vs$ and $s'\in \vs$.

\item $S$ is a  ring if and only if $\vs=S$.
\end{enumerate}
\end{lemma} 

\subsection{Basics}\label{basic}
A \emph{semisubtractive} ideal $\ma$ of a semiring $S$ is an ideal of $S $ such that $x\in \ma\cap  \vs$ implies that $-x\in \ma\cap \vs.$
In other words, every element $x$ in a semisubtractive ideal $\ma$ has its additive inverse in $\ma$. Since by Lemma \ref{epvs}(\ref{vsne}), the set $
\vs$ is nonempty, the above definition indeed makes sense. It is easy to see that the zero ideal of a semiring $S$  is semisubtractive, and is contained in every semisubtractive ideal of $S$. We shall denote by $\ids$ the set of  semisubtractive ideals of $S$. 

Let us also recall  the definitions of two types of ideals of semirings that are closely related to semisubtractive ideals. An ideal $\ma$ of $S$ is called \emph{subtractive} (or \emph{$k$-ideal}) if $x\in \ma$ and $x+y\in \ma$ implies  $y\in \ma$. An ideal $\ma$ if $S$ is called \emph{strongly subtractive} (or \emph{$r$-ideal}) if $x+y\in \ma$ implies $x\in \ma$ and $y\in \ma$. 

\begin{proposition}\label{bpss}
In a semiring $S$, the following hold.
\begin{enumerate}
\item\label{sss} Every strong subtractive ideal and every subtractive ideal of $S$ are semisubtractive.

\item\label{pss} Every prime ideal of $S$ is semisubtractive.

\item\label{ssci} If $\{\ma_{\lambda}\}_{\lambda \in \Lambda}$ are semisubtractive ideals of $S$, so is  $\bigcap_{\lambda \in \Lambda}\ma_{\lambda}$.

\item\label{cuss} If $\{\ma_{\lambda}\}_{\lambda \in \Lambda}$ are semisubtractive ideals of $S$, so is  $\sum_{\lambda \in \Lambda}\ma_{\lambda}$.

\item\label{cufp} If $\ma$ and $\mb$ are semisubtractive ideals of $S$, so is $\ma\mb$.

\item\label{icj} If $\ma$ is a semisubtractive ideal and $\mb$ is an ideal of  $S $, then $(\ma:\mb)$ is a semisubtractive ideal.

\item\label{cicj} Suppose that  $\mb$, $\{\mb_{\lambda}\}_{\lambda \in \Lambda}$, $\mc$ are ideals and  $\ma$, $\{\ma_{\omega}\}_{\omega \in \Omega}$ are semisubtractive ideal of  $S$. Then  $ (\ma\colon \mb),$  $((\ma\colon \mb)\colon \mc),$ $(\ma\colon \mb\mc),$ $((\ma\colon \mc)\colon \mb),$ $ (\bigcap_{\omega} \ma_{\omega}\colon \mb),$ $ \bigcap_{\omega}(\ma_{\omega}\colon \mb),$ $ (\ma\colon \sum_{\lambda}\mb_{\lambda}),$ and $ \bigcap_{\lambda}(\ma\colon \mb_{\lambda})$ are all semisubtractive ideals.

\item\label{assi} If $X$ is a nonempty subset of $S $, then the annihilator $\mathrm{Ann}_S (X)$ of $X$ is a semisubtractive ideal.

\item\label{riss} For every ideal $\ma$ of $S$, the radical $\sqrt{\ma}$ of $\ma$ is a semisubtractive ideal.

\item\label{nrss} Let  \( \mathcal{N}(S) \) be the nilradical of \( S \). If \( \mathfrak{a} \) is a semisubtractive ideal, then so is \( \mathfrak{a} \cap \mathcal{N}(S) \).
\end{enumerate}
\end{proposition}

\begin{proof}
The proof of (\ref{sss}) and (\ref{ssci}) are straightforward.   Proofs of (\ref{pss}) and (\ref{cuss})  may be found respectively in Corollary 7.8 and in Example 6.42 of \cite{Gol99}.
To prove (\ref{cufp}), suppose that \( \ma \) and \( \mb \) are semisubtractive ideals of \( S \). 
Let \( z \in \ma\mb \cap \vs \). Then \( z \) can be written as 
\[
z = \sum_{k=1}^{n} i_k j_k,
\]
for some \( i_k \in \ma \) and \( j_k \in \mb \). Since \( z \in \vs \), there exists \( -z \in S \) such that \( z + (-z) = 0 \). We claim that \( -z \in \ma\mb \cap \vs \).
Consider \( -z = -\sum_{k=1}^{n} i_k  j_k \). Since
\[
-z = \sum_{k=1}^{n} (-i_k  j_k) = \sum_{k=1}^{n} i_k  (-j_k) = \sum_{k=1}^{n} (-i_k)  j_k,
\]
and since \( \ma \) and \( \mb \) are semisubtractive, it follows that \( -i_k \in \ma \) and \( -j_k \in \mb \), which implies that each term \( (-i_k)  j_k \in \ma\mb \) and \( i_k  (-j_k) \in \ma\mb \).
Therefore, \( -z \in \ma\mb \cap \vs \). The proofs of (\ref{icj}) and (\ref{assi}) are trivial, whereas (\ref{cicj}) follows from (\ref{icj}). To show (\ref{riss}), let us first recall from \cite[Theorem 3.11(1)]{Nas18} that
\[\sqrt{\ma}=\bigcap \left\{ \p\in \spec \mid \p\supseteq \ma\right\}.\]
The claim now follows from (\ref{pss}) and (\ref{ssci}). To prove (\ref{nrss}), let \( x \in \mathfrak{a} \cap \mathcal{N}(S) \cap \mathcal{V}(S) \). Since \( x \) is nilpotent, \( x^n = 0 \) for some \( n\in \mathds{N}^+ \), and  since $\ma$ is semisubtractive, \( -x \in \mathfrak{a} \).
As \( -x^n=(-x)^n = 0 \), we conclude that \( -x \in \mathcal{N}(S) \). Hence, \( \mathfrak{a} \cap \mathcal{N}(S) \) is semisubtractive.
\end{proof}

\begin{remark}
It is easy to see that the ideal $\ma:=\mathds{N}\setminus \{1\}$ of the semiring $\mathds{N}$ is semisubtractive, but not subtractive. Also note that the properties (\ref{cuss})  and (\ref{cufp}) may not hold for subtractive ideals (see \cite[Example 6.19]{Gol99} and \cite[Example 6.43]{Gol99} respectively).   
\end{remark}

The lattice of all ideals of a  semiring, in general, is not modular. However,  the lattice of subtractive ideals is so (see \cite[Proposition 6.38]{Gol99}). Interestingly enough, it is also true for semisubtractive ideals.

\begin{theorem}
For every semiring $S$,  $(\ids, \subseteq)$ is  a complete modular lattice.
\end{theorem}

\begin{proof}
The fact that the lattice  is complete follows from (\ref{ssci}) and (\ref{cuss}) of Proposition \ref{bpss}. To show modularity, suppose $\ma$, $\mb\in \ids$ and $\ma\subseteq \mb$. It is sufficient to show that for any $\mc\in \ids$, we have
\[(\ma+\mc)\cap \mb \subseteq \ma+(\mc\cap \mb).\]
Let $x\in (\ma+\mc)\cap \mb$. Then $x\in \mb$ and $x=s+t\in \ma+\mc$, for some $s\in \ma$ and $t\in \mc$. Since by hypothesis, $\ma\subseteq \mb$, we have $s\in \mb$. Since $\mb\in \ids,$ we get $-s\in \mb$. Therefore, $t=-s+(s+t)\in \mb$, and that implies $t\in \mc\cap \mb.$ Hence $x=s+t\in \ma+(\mc\cap \mb)$.
\end{proof}

It is shown in \cite[Corollary 2.2]{SA93} that in a commutative semiring with identity, every proper subtractive ideal is contained in a maximal subtractive ideal. The same holds for semisubtractive ideals, and the proof is essentially same to that for the subtractive case, we will only record the result.

\begin{proposition}\label{mxc}
Every proper semisubtractive ideal of a semiring is contained in a maximal semisubtractive ideal.
\end{proposition}

Our aim now is to study the properties of semisubtractive ideal and their products, intersections, and ideal quotients under  semiring homomorphisms. 
Suppose that  $\phi\colon S \to S '$ is a semiring homomorphism.  If $\mb$ is a semisubtractive ideal of $S '$, then  $\mb^c$ denotes the ideal $\phi\inv (\mb).$

\begin{proposition}\label{cep}
	
Let $\phi\colon S \to S '$ be a semiring homomorphism. For semisubtractive ideals $\mb,$ $\mb_1$, and $\mb_2$  of $S '$, the following hold.
	
\begin{enumerate}
		
\item\label{jcki}  $\mb^c$ is a semisubtractive ideal of  $S$.
		
\item The kernel $\mathrm{ker}\phi$ is a semisubtractive ideal of $S .$	
		
\item\label{crj} $ (\mb_1\cap \mb_2)^c= \mb_1^c\cap \mb_2^c$, $ (\mb_1\mb_2)^c\supseteq \mb_1^c\mb_2^c$, $(\mb_1:\mb_2)^c\subseteq (\mb_1^c:\mb_2^c)$.

\item If \( \phi \) is  surjective,
then  \( \phi(\mathfrak{a}) \) of  a semisubtractive ideal \( \mathfrak{a} \) of $S$ is semisubtractive of \( S' \).
\end{enumerate}	
\end{proposition}

\begin{proof}
(1)--(3) Trivial. For (4), observe that  $\phi(\mathfrak{a})$ is an ideal of $S'$ follows from surjectivity of \( \phi\). Let \( y \in \phi(\mathfrak{a}) \cap \mathcal{V}(S') \), so there exists \( x \in \mathfrak{a} \cap V(S) \) such that \( \phi(x) = y \). Since \( \mathfrak{a} \) is semisubtractive, \( -x \in \mathfrak{a} \), and \( \phi(-x) = -y \). Hence, \( -y \in \phi(\mathfrak{a}) \), proving that \( \phi(\mathfrak{a}) \) is semisubtractive.
\end{proof}

We shall now study a closure operation  on the set of all ideals of a semiring $S$ that will assign a unique  semisubtractive ideal to each ideal of $S$.
There are evidences of defining various types of ideals of  semirings using  closure operators (see \cite{JRT22, SA92} for subtractive ideals and \cite{Han15} for $r$-ideals). 
Our choice of closure operator was originally pointed out by Golan in \cite[Example 6.40]{Gol99}.
For a semiring $S $,  the \emph{Golan closure}  (or, \emph{$G$-closure}) operation on $\id$  is defined by
\begin{equation*}
\label{clkdef}
\cz(\ma):=\bigcap\left\{\mb\in \ids\mid \mb\supseteq \ma\right\}.
\end{equation*} 	

In the next proposition we shall establish some properties of $G$-closure operations.

\begin{proposition}\label{lclk}
In the following, $\ma$, $\{\ma_{\lambda}\}_{\lambda \in \Lambda}$, and $\mb$ are ideals of a semiring $S $. 
\begin{enumerate}
		
\item\label{iclk} $\cz(\ma)$ is the smallest semisubtractive ideal containing $\ma$.
		
\item \label{zssi}
$\cz(0)=0.$
		
\item \label{ckr}
$ \cz(S) =S.$
		
\item\label{clcl} $\cz(\cz(\ma))=\cz(\ma).$
		
\item\label{ijcl} If $\ma\subseteq \mb$, then $\cz(\ma)\subseteq \cz(\mb).$
		
\item \label{clu}
$\cz(\langle \ma\cup \mb\rangle )\supseteq \cz(\ma) \cup \cz(\mb).$
		
\item\label{arbin} $\cz\left(\bigcap_{\lambda\in \Lambda}\ma_{\lambda}\right)=\bigcap_{\lambda\in \Lambda} \cz (\ma_{\lambda}).$
		
\item\label{altd} $\ma$  is a semisubtractive ideal if and only if $\ma=\cz(\ma).$
		
\item\label{cabcc} $\cz(\ma+\mb)=\cz(\cz(\ma)+\cz(\mb)).$
		
\item\label{cgsa} $\cz\left(  \ma\right)\subseteq \sqrt{\ma}$.
\end{enumerate}
\end{proposition}

\begin{proof}
Claim (\ref{iclk}) follows from Proposition \ref{bpss}(\ref{ssci}) and definition of $\cz$. Since $0$ is a semisubtractive ideal, (\ref{zssi}) follows from the definition of $\cz$. The proofs of
(\ref{ckr})--(\ref{arbin}) are straightforward, and (\ref{altd}) follows from (\ref{iclk}).  To show (\ref{cabcc}),  observe that $\ma+\mb\subseteq \cz(\ma+\mb)$ and  $\ma+\mb\subseteq \cz(\cz(\ma)+\cz(\mb)).$ Since by (\ref{iclk}), $\cz(\ma+\mb)$ is the smallest semisubtractive ideal containing $\ma+\mb$, we must have \[\cz(\ma+\mb)\subseteq \cz(\cz(\ma)+\cz(\mb)).\] For the other inclusion, note that $\cz(\ma+\mb)\supseteq \ma,$ $\mb;$ and hence, by (\ref{clcl}) we have $\cz(\ma+\mb)\supseteq \cz(\ma)+ \cz(\mb).$ Applying (\ref{clcl}) and (\ref{ijcl}) in the last containment, we obtain the desired claim. The proof of (\ref{cgsa}) follows from (\ref{ijcl}), (\ref{altd}), and Proposition \ref{bpss}(\ref{riss}).
\end{proof}

\subsection{$Q$-ideals and $\ids$}\label{qiss}

The notion of a $Q$-ideal in a semiring has been introduced in \cite{All69} in order to build
the quotient structure of a semiring with respect to a $Q$-ideal. An ideal $\ma$ of a semiring $S$ is said to be \emph{$Q$-ideal} if there exists a subset $Q$ of $S$ satisfying the following conditions:
\begin{enumerate}
\item[$\bullet$]
$\{q+\ma\}_{q\in Q}$ is a partition of $S$;

\item[$\bullet$] if $q_1$, $q_2\in Q$ such that $q_1\neq q_2$, then $(q_1+\ma)\cap (q_2+\ma)=\emptyset$.
\end{enumerate}
Let us record the following lemma from \cite[Lemma 7]{All69}, which helps to build the quotient structures with respect to  $q$-ideals.

\begin{lemma}
Let $\ma$ be a $Q$-ideal of a  semiring $S$. If $x\in S$, then there
exists a unique $q\in Q$ such that $x+\ma\subseteq q+\ma$.
\end{lemma}

Given a $Q$-ideal $\ma$, let us denote $S/\ma:=\{ q+\ma\mid q\in Q\}$. It has been shown in \cite[Theorem 8]{All69} that $S/\ma$ forms a semiring under the following two binary operations:
\begin{enumerate}
\item[$\bullet$]
$(q_1+\ma) \oplus (q_2+\ma) :=q_3+\ma$, where $q_3$ is the unique element in $Q$
such that $q_l+q_2+\ma \subseteq q_3+\ma$; 

\item[$\bullet$] 
$ (q_1+\ma) \odot (q_2+\ma) :=q_3+\ma$, where $q_3$ is the unique element of $Q$
such that $q_lq_2+\ma\subseteq q_3+\ma$.
\end{enumerate}

All the results in the next proposition are lifted from \cite{Ata07}, and these are the ``semisubtractive version'' of them.  We shall only provide proofs where the semisubtractivity property of ideals are involved. The arguments for other parts are identical to the corresponding results  in \cite{Ata07}.

\begin{proposition}
In the following suppose $S$ is a semiring and $\mi$ is a $Q$-ideal of $S$.
\begin{enumerate}
\item\label{qai} Let  $\ma$ be a semisubtractive ideal
of $S$ with $\mi\subseteq \ma$. If $Q$ is closed under additive inverses, then $\ma/\mi:=\{ q+\mi \mid q\in Q\cap \ma\}$ is a semisubtractive ideal of $S/\mi$.

\item\label{iqs} If $\mc$ is a semisubtractive ideal of
$S/\mi$, then $\mc = \mj/\mi$, for some $\mj\in \ids$ .

\item\label{iqss} Let \( \mi \) be a proper \( Q \)-ideal of  \(S \). If \( S/\mi \) is a semifield, then \( \mi \) is a maximal semisubtractive ideal of \( S \).

\item \label{qiji}
If  $\mi$, $\mj$ are
semisubtractive ideals of $S$, then $(\mi + \mj)/\mi$ is a semisubtractive ideal of $S/\mi$. 

\item \label{pqss} Let  $\p$ be a semisubtractive ideal of
$S$ with $\mi\subseteq \p$. Then $\p$ is a prime ideal of $S$ if and only if $\p/\mi$ is a prime
ideal of $S/\mi$.
\end{enumerate}
\end{proposition}

\begin{proof}
To show $\ma/\mi$ is a semisubtractive ideal of $S/\mi$ in (\ref{qai}), suppose \( z + \mi \in \ma/\mi \cap \mathcal{V}(S/\mi) \). This means that there exists an element \( -z + \mi \in S/\mi \) such that
\[
(z + \mi) + (-z + \mi) = \mi.
\]
Since \( z + \ma \in \ma/\mi \), by definition, we know that \( z \in \ma \cap Q \). Since \( A \) is semisubtractive and $Q$ is closed under additive inverses, we have \( -z \in \ma \cap Q \), which implies that \( -z + \ma \in \ma/\mi \).

For (\ref{iqs}), suppose \( q' \) be the unique element in \( Q \) such that \( q' + \mi \) is the zero element in \( S/\mi \). Define
\[
\mj := \{ r \in S : q + \mi \in \mc, \text{ for some } q \in Q \text{ such that } r \in q + \mi \}.
\]
We wish to show  \( \mj \in \ids\)  and  \( \mc = \mj/\mi \). Suppose \( x \in \mj \cap \vs \). By the definition of \( \mj \),  we have \( x + \mi \in \mc \), and since \( \mc \) is a semisubtractive ideal of \( S/\mi \), we must have \( -x + \mi \in \mc \). Hence, \( -x \in \mj \), showing that \( \mj \) is  semisubtractive.

To obtain (\ref{iqss}), Suppose  \( q' \in Q \) is the unique element in \( Q \) such that \( q' + \mi \) is the zero element of \( S/\mi \), and let \( q_0 + \mi \) be the identity element of \( S/\mi \) by \cite[Lemma 2.4]{Ata07}, where \( 1 = q_0 + a \) for some \( a \in \mi \) and the unique element \( q_0 \in Q \).
Let us assume that \( S/\mi \) be a semifield, and suppose \( \mi \subseteq \mj \) for some semisubtractive ideal \( \mj\) of \( S \). We claim that \( \mj = S \).
There are elements \( t \in Q \) and \( c \in \mj \setminus \mi \) with \( c \in t + \mi \), so \( c = t + d \) for some \( d \in \mi \), which implies \( t \in \mj \setminus \mi \).
If \( t + \mi = q' + \mi = \mi \), then \( t = q' \in \mi \), which is a contradiction. Therefore, there exists \( s + \mi \in S/\mi \) such that
\[
(t + \mi) \odot (s + \mi) = q_0 + \mi.
\]
Thus, \( st + e = q_0 + f \) for some \( e, f \in \mi \), and it follows that
\[
st + a + e = q_0 + a + f = 1 + f \in \mj.
\]
Since \( \mj \) is a semisubtractive ideal, $-f\in \mj$, and hence, \( 1 =-f+(1+f)\in \mj \), implying that \( \mj = S \), as required.

To prove  (\ref{qiji}), observe that by \cite[Lemma 2]{GC06}, every $Q$-ideal is subtractive, and hence, by Proposition \ref{bpss}(\ref{sss}) is semisubtractive. Moreover, by Proposition \ref{bpss}(\ref{cuss}), $\mi+\mj$ is semisubtractive, and hence, $(\mi+\mj)/\mi$ is a semisubtractive ideal of $S/\mi$ by (\ref{qai}).

Finally, to show one half of (\ref{pqss}), suppose that \( \p/\mi \) is a prime ideal of $S/\mi$. Let \( a\), \(b \in S \) such that \( ab \in \p \). Then there are elements \( q_1, q_2 \in Q \) such that \( a \in q_1 + \mi \) and \( b \in q_2 + \mi \), so \( a = q_1 + c \) and \( b = q_2 + d \) for some \( c, d \in \mi \).
Since \[ab = q_1 q_2 + q_1 d + c q_2 + cd \in \p \] and \( \p \) is a semisubtractive ideal of \( S\), we must have $-q_1d$, $-cq_2$, $-cd\in \p$, and hence, \( q_1 q_2 \in \p \).
Let \( q \) be the unique element in \( Q \) such that \[ (q_1 + \mi) \odot (q_2 + \mi) = q + \mi, \] where \( q_1 q_2 + \mi \subseteq q + \mi \). Thus, \( q + e = q_1 q_2 + f \) for some \( e, f \in \mi \).
Since \( \p \) is a semisubtractive ideal of \( S \), and $\mi\subseteq \p$, we have $-e\in \p$. Hence, we get \( q \in Q \cap \p \) and \( q + \mi \in \p/\mi \). Hence \( q_1 + \mi \in \p/\mi \) or \( q_2 + \mi \in \p/\mi \), which implies that \( q_1 \in \p \) (so \( a \in \p \)) or \( q_2 \in \p \) (so \( b \in \p \)). Therefore, \( p\in \spec\).
\end{proof}

\subsection{$s$-local semirings}\label{sls}

By Proposition \ref{mxc}, it is guaranteed that maximal semisubtractive ideals exist.  We say a semiring $S$ is \emph{$s$-local} if $S$ has a unique maximal semisubtractive ideal. Recall that a non-zero element $x$ of $S$ is called \emph{semi-unit} if there exists $s$, $t\in S$ such that $1+sx=tx$. All the results in this section are generalizations of  \cite{AA08}. We start with a basic fact for semi-units.

\begin{lemma}\label{sus}
Suppose $\ma$ is a semisubtractive ideal of a semiring $S$. If $x$ is a semi-unit of $S$ and if $x\in \ma$, then $\ma=S$.
\end{lemma}

\begin{proof}
Since $x$ is a semi-unit, there exists $s$, $t\in S$ such that $1+sx=tx$. Since $x\in \ma$ and $\ma$ is an ideal, we have $tx$, $sx\in \ma$. Since $\ma$ is semisubtractive, we also have $-sx\in \ma$. Therefore, $1=-sx+(1+sx)\in \ma$ implying that $\ma=S$. 
\end{proof}

Next wish to give a characterization of $s$-local semirings in terms of semisubtractive ideals, but first, a lemma.

\begin{lemma}\label{sums}
Let $S$ be a semiring and let $x \in S$. Then $x$ is a semi-unit of $S$ if and
only if $x$ lies outside each maximal semisubtractive ideal of $S$.
\end{lemma}

\begin{proof}
By Lemma \ref{sus}, $x$ is a semi-unit of $S$ if and only if $S = \cz(\langle x\rangle)$. First, suppose	that $x$ is a semi-unit of $S$ and let $x \in \m$ for some maximal semisubtractive ideal $\m$ of $S$.
Then we have $\langle x\rangle \subseteq \m \subset S$, so that $x$ can not be a semi-unit of $S$.
Conversely, if $x$ is not a semi-unit of $S$, then $1 + rx = sx$ holds for no $r,$ $s \in S$.
Hence, $1 \notin \cz(\langle x\rangle)$ yields that $\cz(\langle x\rangle)$ is a proper semisubtractive ideal of $S$ by Proposition \ref{lclk}(\ref{iclk}). By
Proposition \ref{mxc}, $\cz(\langle x\rangle)\subseteq \mj$ for some maximal semisubtractive ideal $\mj$ of $S$; but this would
contradict  that $x$ lies outside each maximal semisubtractive ideal of $S$.
\end{proof}

\begin{theorem}\label{nsu}
A semiring $S$ is $s$-local if and only if the set
of non-semi-unit elements of $S$ forms a semisubtractive ideal.
\end{theorem}

\begin{proof}
Suppose $S$ is an $s$-local semiring with unique maximal semisubtractive ideal $\m$. By Lemma
\ref{sums}, $\m$ is precisely the set of non-semi-unit elements of $S$. For the converse, let us assume that
the set of non-semi-units of $S$ is a semisubtractive ideal $\ma$ of $S$. Then  $\m\neq S$ since $1$ is a semi-unit of
$S$. Since $S$ is not trivial, by Proposition \ref{mxc}, $S$ has a maximal semisubtractive ideal, say $\mj$. By
Lemma \ref{sums}, $\mj$ consists of non-semi-units of $S$, and so $\mj\subseteq \m\subsetneq S$ . Thus $\m = \mj$. This proves the uniqeness.
\end{proof}

Let $S$ be a semiring, and let $X$ be the set of all multiplicatively cancellable
elements of $S$. Following the construction of localization of rings, we can obtain a semiring $S_X$, the localization of $S$ at $X$. If $\ma$ is an ideal of $S$, we shall denote its extension in $S_X$ by $\ma S_X$. 
The next lemma tells preservation of semisubtractive ideals under localizations, and its proof is trivial.

\begin{lemma}\label{psl}
If $\ma$ is a semisubtractive ideal of a semiring $S$, then so is $\ma S_X$  of $S_X$.
\end{lemma}

\begin{proposition}
Let $S$ be an $s$-local semiring with unique maximal semisubtractive ideal $\m$ such that $S \cap \m = \emptyset$. Then $S_X$ is an $s$-local semiring with unique maximal semisubtractive ideal $\m S_X$.
\end{proposition}

\begin{proof}
	To obtain the desired claim, thanks to  Lemma \ref{psl} and Theorem \ref{nsu}, it is sufficient to show that $\m S_X$ is exactly the set of non-semi-units of $S_X$. Suppose $z \in S_X \setminus \m S_X$. Then  $z = \frac{a}{s}$ for some $a \in S$ and $s \in X$. We must have $a \notin \m$, so $1 + xa = ya$, for some $x,$ $y \in S$ by Lemma \ref{sums}. It then follows from the equation
	\[
	\frac{1}{1} + \left(\frac{xs}{1}\right)\left(\frac{a}{s}\right) = \left(\frac{ys}{1}\right)\left(\frac{a}{s}\right)
	\]
	that $\frac{a}{s}$ is a semi-unit of $S_X$.
	On the other hand, if $y$ is a semi-unit of $S_X$, then  $y = \frac{b}{t}$, for some $b \in S$ and $t \in X$, and there exist $c,$ $d \in S$ and $u, w \in X$ such that
	\[
	\frac{1}{1} + \left(\frac{b}{t}\right)\left(\frac{c}{u}\right) = \left(\frac{b}{t}\right)\left(\frac{d}{w}\right).
	\]
	It follows that
	\[
	t^2uw + bctw = tubc.
	\]
	From this, we claim that $b \notin \m$. Otherwise, since $\m$ is semisubtractive, we would get \[t^2uw=(t^2uw + bctw)+(-bctw)\in \m,\] a contradiction by Lemma \ref{sums}. Therefore, it follows that $y \notin \m S_X$, and this completes the proof.
\end{proof}

\subsection{Subclasses of $\ids$}\label{siid}
We shall now introduce two subclasses of semisubtractive ideals and study some of their properties. Let $S $ be a semiring.
A semisubtractive ideal $\ma$ of $S $ is called \emph{$s$-irreducible} if  $\mb\cap \mb'= \ma$ implies that either $\mb= \ma$ or $\mb'= \ma$ for all $\mb,$ $\mb'\in \ids$.
A semisubtractive ideal $\ma$ of $S $ is called \emph{$s$-strongly irreducible} if  $\mb\cap \mb'\subseteq \ma$ implies that either $\mb\subseteq \ma$ or $\mb'\subseteq \ma$ for all $\mb,$ $\mb'\in \ids$. 

Our next result tells that whenever we want to study semisubtractive ideals that are also irreducible (or strongly irreducible), it is sufficient to study $s$-irreducible (or $s$-strongly irreducible) ideals. 

\begin{theorem}\label{eqsi}
An ideal $\mc$ of a semiring $S$ is $s$-irreducible ($s$-strongly irreducible)  if and only if $\mc$ is   irreducible (strongly irreducible)  as well as a semisubtractive ideal of  $S $.
\end{theorem}

\begin{proof}
We give a proof for $s$-strongly irreducible ideals, that for $s$-irreducible ideals requiring only a
trivial change of terminology.
Suppose that $\mc$ is a $s$-strongly irreducible ideal and $\ma$, $\mb$ are ideals of $S $ such that $\ma\cap \mb\subseteq \mc.$ This implies
\[ \cz(\ma)\cap \cz(\mb)=\cz(\ma\cap \mb)\subseteq \cz(\mc)=\mc,\]
where, the first equality follows from Proposition \ref{lclk}(\ref{arbin}) and the inclusion from Proposition \ref{lclk}(\ref{ijcl}). By hypothesis, this implies that either $ \cz(\ma)\subseteq \mc$ or $ \cz(\mb)\subseteq \mc$. Since by Proposition \ref{lclk}(\ref{iclk}), $\ma\subseteq \cz(\ma)$ and $\mb\subseteq \cz(\mb)$ for all $\ma$, $\mb\in \id,$ we have the desired claim. The  converse claim is obvious.
\end{proof}

Likewise (commutative) rings, it is known (see \cite[Proposition 7.33]{Gol99}) that a strongly irreducible ideal $\mc$ of a (commutative) semiring $S$ has the following equivalent `elementwise' property: If $x$, $y \in S$ satisfy $\langle x\rangle \cap \langle y\rangle \subseteq \mc$, then either $x\in \mc$ or $y\in \mc$. The next proposition shows that a similar result holds for $s$-strongly irreducible ideals.

\begin{proposition}
\label{abi} %Golan 7.33
An ideal $\mc$ of a semiring $S $ is  $s$-strongly irreducible  if and only if	for all $x$, $y\in S$ satisfy $\cz(\langle x\rangle)  \bigcap \cz(\langle y\rangle) \subseteq  \mc$ implies either $x\in \mc$ or $y\in \mc$.
\end{proposition}

\begin{proof}
Suppose that $\mc$ is a $s$-strongly irreducible ideal of $S$ and $\langle x\rangle  \cap \langle y\rangle \subseteq  \mc$ for all $x,$ $y\in S .$ This implies
\[\cz(\langle x\rangle \cap \langle y\rangle)=\cz(\langle x\rangle)\cap  \cz(\langle y\rangle) \subseteq \cz (\mc)=\mc\]
where the first equality follows from Proposition \ref{lclk}(\ref{arbin}).
Since $\mc$ is  $s$-strongly irreducible, we must have either $x\in \langle x\rangle\subseteq  \cz(\langle x\rangle) \subseteq \mc$ or  $y\in \langle y\rangle\subseteq  \cz(\langle y\rangle) \subseteq \mc$.

To show the converse, suppose that $\mc$ is a semisubtractive ideal which is not $s$-strongly irreducible. Then there are semisubtractive ideal $\ma$ and $\mb$ satisfy $\ma\cap \mb\subseteq \mc,$ but $\ma\nsubseteq \mc$ and $\mb\nsubseteq \mc.$ This implies there exist $x\in \ma$ and $y\in \mb$ such that $\cz(\langle x\rangle)  \cap \cz(\langle y\rangle) \subseteq  \mc$, but neither $x\notin \mc$ nor $y\in \mc$, a contradiction.
\end{proof}

The rest of the results in this section are extensions of the corresponding results  on strongly irreducible ideals in rings from  \cite{Azi08}.
We shall now see that every proper semisubtractive ideal can be expressed  in terms of $s$-irreducible ideals. But first, a lemma.

\begin{lemma}\label{lir}
%Golan 7.34
Suppose $S$ is a semiring. Let $0\neq x\in S$ and $\mc$	be a proper  semisubtractive ideal of  $S$ such that $x\notin \mc.$ Then there exists a $s$-irreducible ideal $\mb$ of $S $ such that $\mc\subseteq \mb$
and $x\notin \mb$.	
\end{lemma}

\begin{proof}
Let
$\{\mb_{\lambda}\}_{\lambda \in \Lambda}$ be a chain of semisubtractive ideal of $S$ such that $x\notin \mb_{\lambda}\supseteq \mc$ for all $\lambda \in \Lambda$. Then
\[x\notin \bigcup_{\lambda \in \Lambda} \mb_{\lambda}\supseteq \mc.
\]
By Zorn's lemma, there exists a maximal element $\mb$ of this chain. Suppose that  $\mb=\mb_1\cap \mb_2.$ By the maximality condition of $\mb$, we must have $x\in \mb_1$ and $x\in \mb_2,$ and hence $x\in \mb_1\cap \mb_2=\mb,$ a contradiction. Therefore, $\mb$ is the required $s$-irreducible ideal.
\end{proof}

\begin{proposition}
%Golan 7.35
If $\mc$ is a proper subtractive ideal of a semiring $S$, then $\mc=\bigcap_{\mc\subseteq \mb} \mb$, where $\mb$ is a $s$-irreducible ideal of  $S$.
\end{proposition}

\begin{proof}
By Lemma \ref{lir}, there exists a $s$-irreducible ideal $\mb$ of  $S $ such that $\mc\subseteq \mb$. Suppose that  \[\mb'=\bigcap_{\mc\subseteq \mb} \{\mb\mid \mb\;\text{is $s$-irreducible}\}.\]
Then $\mc\subseteq \mb'$. We claim that $\mb'=\mc.$   If $\mb'\neq \mc$, then there exists an $x\in \mb'\setminus \mc$, and by Lemma \ref{lir}, there exists a $s$-irreducible ideal $\mb''$ such that $x\notin \mb''\supseteq \mc,$ a contradiction.
\end{proof}

For a weakly Noetherian semiring,  we even have a stronger result.

\begin{proposition}
%Nasehpour2018, Proposition 7.3
Every semisubtractive ideal of a weakly Noetherian semiring $S$ can be represented as an intersection of a finite number of $s$-irreducible ideals of $S$.
\end{proposition}

\begin{proof}
Suppose 
\[\mathcal{F}:=\left\{\ma\in  \ids \mid \ma\neq \bigcap_{i=1}^n \mc_i,\;\mc_i \;\text{is $s$-irreducible}\right\}.\]  It is sufficient to show that $\mathcal{F}=\emptyset$.
Since $S$ is weakly Noetherian, $\mathcal{F}$ has a maximal element $\mc$. Since $\mc\in \mathcal{F}$, it is not a finite
intersection of $s$-irreducible ideals of $R $. This implies that $\mc$ is not $s$-irreducible. Hence, there are semisubtractive ideal $\ms$ and $\mt$ such that  $\ms\supsetneq \mc$, $\mt\supsetneq \mc$, and $\mc = \ms \cap \mt.$ Since $\mc$ is
a maximal element of $\mathcal{F}$, we must have $\ms,$ $\mt \notin \mathcal{F}.$ Therefore, $\ms$ and $\mt$ are a finite intersection of
$s$-irreducible ideals of $S$, which subsequently implies that $\mc $ is also a finite intersection of $s$-irreducible
ideals of $S$, a contradiction.
\end{proof}

In Proposition \ref{mxc}, we have seen that every proper semisubtractive ideal is contained in a maximal ideal of this type. The next result shows a similar containment, but with a minimal $s$-strongly irreducible ideal.

\begin{proposition}
%Azizi Th 2.1
Every proper semisubtractive ideal of a semiring is contained in a minimal $s$-strongly irreducible ideal.
\end{proposition}

\begin{proof}
Let $\ma$ be a proper semisubtractive ideal of a semiring $S$. Consider the chain $(\mathcal{E}, \subseteq)$, where \[\mathcal{E}:=\{\mb\mid \ma\subseteq \mb,\;\mb\;\text{is $s$-strongly irreducible}\}.\]
We claim that $\mathcal{E}\neq \emptyset$. Indeed:  every proper semisubtractive ideal is contained in a maximal  ideal $\m$, which is prime, and hence, a strongly irreducible ideal. However, by Proposition \ref{bpss}(\ref{pss}), $\m$ is also semisubtractive. By Theorem \ref{eqsi}, $\m$ is $s$-strongly irreducible. Hence $\m\in \mathcal{E}$. Therefore, the set $\mathcal{E}$ is nonempty. By Zorn's lemma, $\mathcal{E}$ has a minimal element, which is our desired minimal $s$-strongly irreducible ideal.
\end{proof}

The next result characterizes semirings in which $s$-irreducibility and $s$-strongly irreducibility coincide.

\begin{theorem}
%Iseki, ideal theory of semirings, Th. 7
In a arithmetic semiring $S $, a semisubtractive ideal of a semiring $S$ is
$s$-irreducible if and only if it is $s$-strongly irreducible. Conversely, 
if a $s$-irreducible ideal of a semiring $S $ is $s$-strongly
irreducible, then $S$ is arithmetic.
\end{theorem}

\begin{proof}
By \cite[Theorem 3]{Is{e}56}, in an arithmetic semiring, irreducible and strongly irreducible ideals are equivalent. Thanks to Proposition \ref{eqsi}, we have then our claim. For the converse, \cite[Theorem 7]{Is{e}56} says that if irreducibility implies strongly irreducibility, then the semiring is arithmetic. Once again, applying Proposition \ref{eqsi} on this result, we obtain the claim.
\end{proof}

\begin{corollary}
%Iseki, ideal theory of semirings, Cor. 2
In an arithmetic semiring, every semisubtractive ideal is the intersection of the $s$-strongly irreducible ideals that contain it.
\end{corollary}

\subsection{Semisubtractive spaces}\label{sssp}

From Proposition \ref{lclk}, it is clear that the sets of the form: \[\h(\ma):=\{ \mb\in \ids\mid \mb\supseteq \ma\}\qquad (\ma\in \ids)\] are not closed under finite unions, and hence, $\h$ is not a Kuratowski closure operator. However, the collection  $\{\h(\ma)\}_{\ma\in \ids}$ of subsets of $\ids$ as subbasic closed sets induce a topology on $\ids$,  which we call a \emph{semisubtractive topology},  denote by $\tau_s$. The set $\ids$ endowed with a semisubtractive topology, will be called a \emph{semisubtractive space}, and instead of $(\ids, \tau_s)$, we shall simply denote the space by $\ids$. Note that our semisubtractive topology is a slightly  modified version of the ideal topology (also known as the coarse lower topology) introduced in \cite{DG23}.

We now study some separation properties of semisubtractive spaces. Since the specialization order of a semisubtractive space forms a partial order, we derive the following separation axiom.

\begin{proposition}\label{t0}
Every semisubtractive space is $T_{{\scriptscriptstyle 0}}.$
\end{proposition}

Next, we aim to look into the sobriety of semisubtractive spaces. Recall that a nonempty closed subset $D$ of a topological space is \emph{irreducible} if $D\neq D_{ 1}\cup D_{ 2}$ for any two proper closed subsets  $D_{ 1}$ and $D_{2}$ of $D$. A point $x$ in a closed subset $D$ is called a \emph{generic point} of $D$ if $D = \overline{\{x\}},$ and a topological space is called \emph{sober} if every nonempty irreducible closed set has a unique generic point.  To obtain the sobriety of a semisubtractive space, we use the following lemma.

\begin{lemma}\label{scir}
Every nonempty subbasic closed set of a semisubtractive space is irreducible.
\end{lemma}

\begin{proof}
To obtain the desired claim, observe that it is sufficient to prove: $\h(\ma)=\overline{\{\ma\}}$ for all $\ma\in \ids.$ Since $\overline{\{\ma\}}$ is the smallest closed set containing $\ma$, it follows that $\h(\ma)\supseteq\overline{\{\ma\}}$. For the other inclusion, first consider the trivial case of $\overline{\{\ma\}}=\ids.$ For this we have \[\ids=\overline{\{\ma\}}\subseteq \h(\ma)\subseteq \ids,\] and hence $\h(\ma)\subseteq \overline{\{\ma\}}.$ Now suppose that $\overline{\{\ma\}}\neq\ids$ and \[\overline{\{\ma\}}=\bigcap_{\lambda \in \Lambda}\left(\bigcup_{i=1}^{n_{\lambda}}\h (\ma_{i\lambda})  \right).\]
This means that $\ma\in \bigcup_{ i=1}^{n_{\lambda}} \h (\ma_{i\lambda})$ for some $i$ and each $\lambda \in \Lambda$. But from this, we obtain that  
\[\h (\ma) \subseteq \bigcap_{\lambda \in \Lambda}\left(\bigcup_{i=1}^{n_{\lambda}}\h (\ma_{i\lambda})  \right),\] and hence we have the desired inclusion.
\end{proof}

\begin{theorem}
Every semisubtractive space is sober.
\end{theorem}

\begin{proof}
Suppose $D$ is a nonempty irreducible closed subset of a semisubtractive space $S$. Then $D=\bigcap_{\lambda \in \Lambda} \mathcal{E}_{\lambda}$, where each $\mathcal{E}_{\lambda}$ is a finite union of subbasic closed sets of $\tau_s.$ Since $D$ is irreducible, for every $\lambda \in \Lambda$, there exists an $\ma_{\lambda}\in \id$ such that 	\[D\subseteq \cz(\ma_{\lambda})\subseteq \mathcal{E}_{\lambda},\]
and this implies
\[D=\bigcap_{\lambda \in \Lambda}\h(\ma_{\lambda})=\h\left( \bigcap_{\lambda\in \Lambda} \ma_{\lambda} \right)=\overline{\left\{\bigcap_{\lambda\in \Lambda}\ma_{\lambda}\right\}},\]
where, the last equality follows from Lemma \ref{scir}. This proves the existence of the generic point, whereas the uniqueness of it follows from Proposition \ref{t0}.
\end{proof}

\begin{proposition}
Every semisubtractive space is connected.
\end{proposition}

\begin{proof}
Since $\h(0)=\ids$ and since irreducibility implies connectedness, the claim now follows from Lemma \ref{scir}.
\end{proof}

Using the Alexander subbase theorem, we now show the quasi-compactness of semisubtractive spaces. For this, we assume that $\h(\ma)=\emptyset$ implies $\ma=S$, for every semiring $S$ and every ideal $\ma\in S$.

\begin{proposition}\label{comp} 
Every semisubtractive space is quasi-compact. 
\end{proposition} 

\begin{proof}   
Let  $\{K_{ \lambda}\}_{\lambda \in \Lambda}$ be a family of subbasic closed sets of a semisubtractive space $\ids$   such that $\bigcap_{\lambda\in \Lambda}K_{ \lambda}=\emptyset.$ Let $\{\ma_{ \lambda}\}_{\lambda \in \Lambda}$ be a family of ideals of $S$ such  that $\forall \lambda \in \Lambda,$ we have $K_{ \lambda}=\h(\ma_{ \lambda}).$  Since \[\bigcap_{\lambda \in \Lambda}\h(\ma_{ \lambda})=\h\left(\sum_{\lambda \in \Lambda}\ma_{ \lambda}\right),\] we get  $\h\left(\sum_{\lambda \in \Lambda}\ma_{ \lambda}\right)=\emptyset,$ and that  implies $ \sum_{\lambda \in \Lambda}\ma_{ \lambda}=S.$ Then, in particular, we obtain $1=\sum_{\lambda_i\in \Lambda}x_{ \lambda_i},$ where $x_{ \lambda_i}\in \ma_{\lambda_i}$ and $x_{ \lambda_i}\neq 0$ for $i=1, \ldots, n$. This shows    $S=\sum_{  i \, =1}^{ n}\ma_{\lambda_i}.$ Therefore,   $\bigcap_{ i\,=1}^{ n}K_{ \alpha_i}=\emptyset,$ and hence by the Alexander subbase theorem, $\ids$ is quasi-compact.  
\end{proof}  

Next, we study continuous maps between semisubtractive spaces. Thanks to Proposition \ref{cep}(\ref{jcki}), we will see that every semiring homomorphism induces a canonical continuous map between the corresponding semisubtractive spaces.

\begin{proposition}\label{conmap}
Let $\phi\colon S\to T$ be a semiring homomorphism  and $\mb\in\idss.$

\begin{enumerate}
		
\item\label{contxr} The induced map $\phi_*\colon  \idss\to \ids$ defined by  $\phi_*(\mb)=\phi\inv(\mb)$ is    continuous.
		
\item\label{shke} If $\phi$ is  surjective, then the semisubtractive space $\idss$ is homeomorphic to the closed subspace $\h(\mathrm{Ker}\phi)$ of the semisubtractive space $\ids.$	

\item\label{den} The subset  $\phi_*(\idss)$ is dense in $\ids$ if and only if  $\mathrm{Ker}\phi\subseteq \bigcap_{\ma\in \ids}\ma.$
\end{enumerate}
\end{proposition}

\begin{proof}      
To show (\ref{contxr}), let $\ma\in \ids$ and $\h(\ma)$ be a   subbasic closed set of the lower space $\ids.$ Then  
\begin{align*}
\phi_*\inv(\h(\ma)) &=\left\{ \mb\in  \idss\mid \phi\inv(\mb)\in \h(\ma)\right\}\\&=\left\{\mb\in \idss\mid \phi(\ma)\subseteq  \mb\right\}\\&=\h(\phi(\ma)), 
\end{align*} 
and hence, the map $\phi_*$  continuous.  
For (\ref{shke}), observe that $\mathrm{Ker}\phi\subseteq  \phi\inv(\mb)$ follows from the fact that  $0\subseteq  \mb$ for all $\mb\in \idss.$ It can thus been seen that $\phi_*(\mb)\in \h(\mathrm{Ker}\phi),$ and hence $\mathrm{Im}\phi_*=\h(\mathrm{Ker}\phi).$  
If $\mb\in \idss,$ then
\[\phi\left(\phi_*\left(\mb\right)\right)=\phi\left(\phi\inv\left(\mb\right)\right)=\mb.\]
Thus $\phi_*$ is injective. To show that $\phi_*$ is closed,  observe that for any   subbasic closed set  $\h(\ma)$ of  $\idss$, we have
\begin{align*}
\phi_*\left(\h(\ma)\right)&=  \phi\inv\left(\h(\ma)\right)\\&=\phi\inv\left\{ \mb\in \idss\mid \ma\subseteq   \mb\right\}\\&=\h(\phi\inv(\ma)). 
\end{align*}
Now if $C$ is a closed subset of $\idss$ and $C=\bigcap_{ \omega \in \Omega} \left(\bigcup_{ i \,= 1}^{ n_{\omega}} \h(x_{ i\omega})\right),$ then
\begin{align*}
\phi_*(C)&=\phi\inv \left(\bigcap_{ \omega \in \Omega} \left(\bigcup_{ i = 1}^{ n_{\omega}} \h(x_{ i\omega})\right)\right)\\&=\bigcap_{ \omega \in \Omega}\, \bigcup_{ i = 1}^{n_{\omega}} \phi\inv\left(\h(x_{ i\omega})\right)
\end{align*}
a closed subset of  $\ids.$ Since by (\ref{contxr}), $\phi_*$ is continuous, we have the desired claim.
Finally, to prove (\ref{den}), 
observe that \[\overline{{\phi_*\left(\idss\right)}}=\phi_*(\h(0))=\h\left(\phi\inv(0)\right)=\h\left(\mathrm{Ker}\phi\right).\] 
It is now sufficient to show that $\h(\mathrm{Ker}\phi)$ to be equal to $\ids$ if and only if  the given condition is satisfied. 
Assume \( \h(\mathrm{Ker}\phi) = \ids \). 
It follows that every \( \ma \in \ids \) satisfies \( k \in \ma \) for all \( k \in \mathrm{Ker}\phi \). Thus
$k \in \ma$  for all $x \in \ids$ and $k \in \mathrm{Ker}\phi.$
From that, we obtain
$
k \in \bigcap_{\ma \in \ids} \ma$, for every $k \in \mathrm{Ker}\phi.
$

Now assume that \( k \in \bigcap_{\ma \in \ids} \ma \) for every \( k \in \mathrm{Ker}\phi \).
Given \( k \in \bigcap_{\ma \in \ids} \ma \), it follows that
$
k \in \ma$ for every  $\ma \in \ids.$
Thus, each \( \ma \in \ids \) contains all elements of \( \mathrm{Ker}\phi \), which implies that
$
\ma \in \h(\mathrm{Ker}\phi)$ for every  $\ma \in \ids.$
Since this is true for all \( \ma \in \ids \), we must have
$
\h(\mathrm{Ker}\phi) = \ids.
$
\end{proof}  

\subsection{$s$-congruences}\label{scon}

Motivated by \cite{Han21}, our final aim is to define and study a congruence related to a semisubtractive ideal.
Recall that	a \emph{congruence} on a semiring \( S \) is an equivalence relation \( \sim \) that is compatible with the semiring operations, that is, for all  $x$, $y$, $z \in S$:
\begin{enumerate}
\item[$\bullet$] \( x \sim y \Rightarrow x + z \sim y + z \),

\item[$\bullet$] \( x \sim y \Rightarrow xz \sim yz \).
\end{enumerate}
Given a semisubtractive ideal \( \mathfrak{a} \) of \( S \),  we shall now construct a congruence corresponding to \( \mathfrak{a} \).
Define a relation \( \sim_{\mathfrak{a}} \) on \( S \) by
\[
x \sim_{\mathfrak{a}} y \Leftrightarrow x - y \in \mathfrak{a} \cap \vs.
\]
The following proposition confirms that \( \sim_{\mathfrak{a}} \) is indeed a congruence.

\begin{proposition}
The relation \( \sim_{\mathfrak{a}} \) is an equivalence relation and is compatible with both the addition and multiplication of the semiring.
\end{proposition}

\begin{proof}
For any \( x \in S \), it is obvious that
$ 
x - x = 0 \in \mathfrak{a} \cap \vs
$. Thus, \( x \sim_{\mathfrak{a}} x \). 	If \( x\sim_{\mathfrak{a}} y \), then \( x - y \in \mathfrak{a} \cap \vs \), so there exists \( z \in S \) such that \( z + (x - y) = 0 \), that is, \( z = y - x \in \mathfrak{a} \cap \vs\). Hence, \( y \sim_{\mathfrak{a}} x \). Suppose \( x \sim_{\mathfrak{a}} y \) and \( y \sim_{\mathfrak{a}} z \). Then
$
x - y \in \mathfrak{a} \cap \vs $ and $y - z \in \mathfrak{a} \cap \vs.
$
Adding these gives
\[
(x - y) + (y - z) = x - z \in \mathfrak{a} \cap \vs,
\]
so \( x \sim_{\mathfrak{a}} z \).
Thus, \( \sim_{\mathfrak{a}} \) is an equivalence relation.
Suppose \( x \sim_{\mathfrak{a}} y \), that is, \( x - y \in \mathfrak{a} \cap \vs\). For any \( z \in S \),
\[
(x + z) - (y + z) = (x - y) \in \mathfrak{a} \cap \vs,
\]
so \( x + z \sim_{\mathfrak{a}} y + z \). Thus, \( \sim_{\mathfrak{a}} \) is compatible with addition.
Now suppose \( x \sim_{\mathfrak{a}} y \), that is, \( x - y \in \mathfrak{a} \cap \vs\). For any \( z \in S \),
$
xz - yz = (x - y)z.
$
Since \( \mathfrak{a} \) is an ideal, \( (x - y)z \in \mathfrak{a} \). Moreover, because \( x - y \in \vs\) and \( S \) is a semiring, we must have \( (x - y)z \in \vs \). Thus, \( xz \sim_{\mathfrak{a}} yz \), and \( \sim_{\mathfrak{a}} \) is compatible with multiplication.
Therefore, \( \sim_{\mathfrak{a}} \) is a congruence on \( S \).
\end{proof}

The congruence defined with respect to a semisubtractive ideal will be called \emph{$s$-congruence}.

\begin{theorem}
There is a bijection between the set of semisubtractive ideals of a semiring \( S \) and the set of $s$-congruences on \( S \).	
\end{theorem}	

\begin{proof}
Let \( \cs \) be the set of all $s$-congruences on \( S \). 
We define a map \( \Phi\colon \ids \to \cs \) by \( \Phi(\mathfrak{a}) := \sim_{\mathfrak{a}} \).
Given an $s$-congruence \( \sim\) in \(\cs \), define the corresponding semisubtractive ideal \( \mathfrak{a}_{\sim} \) by
\[
\mathfrak{a}_{\sim} := \{ x \in S \mid x \sim 0 \}.
	\]
We claim  that \( \mathfrak{a}_{\sim} \) is a semisubtractive ideal. Indeed:
\begin{itemize}
\item[$\bullet$] If \( x \sim 0 \) and \( y \sim 0 \), then \( x + y \sim 0 \) and \( sx \sim 0 \) for all $s \in S$. Hence \( \mathfrak{a}_{\sim} \) is an ideal.

\item[$\bullet$] If \( x \in \mathfrak{a}_{\sim} \cap \vs \), then \( x \sim 0 \) and \( -x \sim 0 \), so \( -x \in \mathfrak{a}_{\sim} \). So, \( \mathfrak{a}_{\sim} \) is semisubtractive.
\end{itemize}
Thus, we define a  map \( \Psi\colon \cs \to \ids \) by \( \Psi(\sim) = \mathfrak{a}_{\sim} \).
Let us start with a semisubtractive ideal \( \mathfrak{a} \in \ids \). Under the map \( \Phi \), we associate the $s$-congruence \( \sim_{\mathfrak{a}} \) defined by
\[
x \sim_{\mathfrak{a}} y \Leftrightarrow x - y \in \mathfrak{a} \cap \vs.
\]
Apply \( \Psi \) to this $s$-congruence and the resulting semisubtractive ideal is
\[
\mathfrak{a}_{\sim_{\mathfrak{a}}} = \{ x \in S \mid x \sim_{\mathfrak{a}} 0 \} = \{ x \in S \mid x \in \mathfrak{a} \cap \vs \}.
\]
But this set is exactly \( \mathfrak{a} \) because by definition, \( \mathfrak{a} \cap \vs \) contains all the elements in \( \mathfrak{a} \) with additive inverses, and by the semisubtractive property, all such elements \( x \in \mathfrak{a} \) must also have \( -x \in \mathfrak{a} \). Hence,
\[
\Psi(\Phi(\mathfrak{a})) = \mathfrak{a}.
\]
Now start with an $s$-congruence $ \sim$ in \(\cs \). Under the map \( \Psi \), we associate the semisubtractive ideal \( \mathfrak{a}_{\sim} := \{ x \in S \mid x \sim 0 \} \). Now, apply \( \Phi \) to this semisubtractive ideal. The resulting $s$-congruence is \( \sim_{\mathfrak{a}_{\sim}} \), defined by
\[
x \sim_{\mathfrak{a}_{\sim}} y \Leftrightarrow x - y \in \mathfrak{a}_{\sim} \cap \vs.
\]
But \( \mathfrak{a}_{\sim} = \{ x \in S \mid x \sim 0 \} \), so \( x \sim_{\mathfrak{a}_{\sim}} y \) if and only if \( x - y \sim 0 \), that is, \( x \sim y \). Hence,
$
\Phi(\Psi(\sim)) = \sim.
$
Therefore, the maps \( \Phi \) and \( \Psi \) are inverses of each other. 
\end{proof}

\subsection*{Concluding remarks}\label{ccr}

In this paper, we have explored how the theory of various subclasses of semisubtractive ideals can be generalized. However, many other aspects remain unaddressed.

Extensive structural aspects of strongly irreducible ideals of rings have been developed in \cite{HRR02}, and some of them have been extended to strongly irreducible 
$k$-ideals in \cite{AA08}. It would be interesting to investigate how far these results can be further generalized to 
$s$-strongly irreducible ideals as defined in this paper. In \cite{S16}, various topologies have been considered on strongly irreducible ideals, and these ideals have also been studied in the context of truncated valuation rings. Naturally, it would be worthwhile to extend these studies to 
$s$-strongly irreducible ideals as well.

In a series of papers \cite{Dal76, Dal762, Dal77, Dal79, AA94}, 
$k$-ideals have been studied in connection with monic and monic free ideals of polynomial semirings, and polynomial semirings in several variables. Once again, it is expected that these results can be generalized to semisubtractive ideals.
 
%\section*{Acknowledgement}

\subsection*{Compliance with Ethical Standards}
The following are not applicable with respect to this article.
\begin{enumerate}
	\item Disclosure of potential conflicts of interest.
	
	\item Research involving human participants and/or animals.
	
	\item Informed consent.
	
	\item Funds received.
	
	\item Data availability.
\end{enumerate}


\begin{thebibliography}{100}
	
\bibitem{All69}
P. J. Allen, A fundamental theorem of homomorphisms for semirings, \textit{Proc. Amer. Math. Soc.}, \textbf{21}(2) (1969), 412--416.

\bibitem{Ata07}
S. E. Atani, The ideal theory in quotients of commutative rings, \textit{Glas. Mat. Ser. III}, \textbf{42}(62)
(2007), 301--308.
	
\bibitem{Ata08}
\bysame, On primal and weakly primal ideals over commutative semirings, \textit{Glas. Mat. Ser. III}, \textbf{43}(63)
(2008), 13--23.

\bibitem{Azi08} A. Azizi, Strongly irreducible ideals, \emph{J. Aust. Math. Soc.}, \textbf{84} (2008), 145--154.

\bibitem{AA94}
F. Alac\'{o}n and D. D. Anderson, The lattice of ideals of a polynomial semiring, \textit{Algebra Univers}, \textbf{31} (1994), 147--149. 

\bibitem{AA194}
\bysame, Commutative semirings and their lattices of ideals, \textit{Houston J. Math.}, \textbf{20}(4) (1994), 571--590.
	
\bibitem{AA08}
R. E. Atani and S. E. Atani, Ideal theory in commutative semirings, \textit{Bul. Acad. Ştiinţe Repub. Mold. Mat.}, (2) (2008), 14--23.

\bibitem{Dal76}
L. Dale, Monic and monic free ideals ina  polynomial semiring, \textit{Proc. Amer. Math. Soc.}, \textbf{56} (1976), 45--50.

\bibitem{Dal762}
\bysame, Monic and monic free ideals in a polynomial semiring in several variables, \textit{Proc. Amer. Math. Soc.}, \textbf{61}(2) (1976), 209--216.

\bibitem{Dal77}
\bysame, The $k$-closure of monic and monic free ideals in a polynomial semiring, \textit{Proc. Amer. Math. Soc.}, \textbf{64}(2) (1977), 219--226.

\bibitem{Dal79}
\bysame, Ideal types in a polynomial halfring, \textit{Proc. Amer. Math. Soc.}, \textbf{75}(2) (1979), 189--195.

\bibitem{DG23}
T. Dube and A. Goswami, Ideal spaces: An extension of structure spaces of rings,\,  \textit{J. Algebra Appl.},  \textbf{22}(11) (2023), 2350245 (18 pages).

\bibitem{DG24}
\bysame, Some aspects of $k$-ideals of semirings, \textit{Rend. Circ. Mat. Palermo (2)}, (2024) (online ready).

\bibitem{DS75}
R. E. Dover and H. E. Stone, On semisubtractive halfrings, \textit{Bull. Austral. Math. Soc.}, \textbf{12} (1975), 371--378.

\bibitem{Gol99}
J. S. Golan, \textit{Semirings and their applications}, Springer, 1999.

\bibitem{GC06}
V. Gupta and J. N. Chaudhari, Right $\pi$-regular semirings, \textit{Sarajevo J.  Math.}, \textbf{14}(2)
(2006), 3--9.

\bibitem{Han15}
H-C. Han, Maximal subtractive ideal and $r $-ideals in semirings, \textit{J. Algebra Appl.},
\textbf{14}(10) (2015), 1250195 (13 pages).

\bibitem{Han21}
\bysame, $k$-congruences and the Zariski topology in
semirings, \textit{Hacet. J. Math. Stat.},
\textbf{50}(3) (2021), 699--709.

\bibitem{Hen58}
M. Henriksen, Ideals in semirings with commutative addition, \textit{Notices Amer. Math. Soc.}, \textbf{6}(3) 31 (1958), 321.

\bibitem{HRR02} W. J. Heinzer, L. J. Ratliff Jr., and D. E. Rush, Strongly irreducible ideals of a commutative ring, \emph{J. Pure Appl. Algebra}, \textbf{166} (2002), no. 3, 267--275.    

\bibitem{Is{e}56}
K. Is\'{e}ki, Ideal theory of semiring, \textit{Proc. Japan. Acad.}, \textbf{32} (1956), 554--559.

\bibitem{JRT22}
J. Jun, S. Ray, and J. Tolliver, Lattices, spectral spaces, and closure operations on
idempotent semirings, \textit{J. Algebra}, \textbf{594} (2022), 313--363.

\bibitem{LaT65}
D. R. LaTorre, On $h$-ideals and $k$-ideals in hemirings, \textit{Publ. Math. Debrecen}, \textbf{12} (1965),
219--226.

\bibitem{Lat67}
\bysame, A note on the Jacobson radical of a hemiring, \textit{Publ. Math. Debrecen}, \textbf{14}
(1967), 9--13.

\bibitem{Les15}
P. Lescot, Prime and primary ideals in semirings, \textit{Osaka J. Math.}, \textbf{52}(3) (2015), 721--737.

\bibitem{Mos70}
J. R. Mosher, Generalized quotients in hemirings, \textit{Compositio Mathematica}, \textbf{22}(3) (1970), 275--281.

\bibitem{Nas18} P. Nasehpour, Some remarks on ideals of commutative semirings, \textit{Quasigroups Related Systems}, \textbf{26} (2018), 281--298.

\bibitem{Ols78}
D. M. Olson, A note on the homomorphism theorem for hemirings, \textit{Internal. J. Math. \& Math. Sci.}, \textbf{1} (1978), 439--445.

\bibitem{S16} N. Schwartz, Strongly irreducible ideals and truncated valuations, \emph{Comm. Algebra}, \textbf{44} (2016), no.3, 1055--1087.

\bibitem{Sto72}
H. E. Stone, Ideals in halfrings, \textit{Proc. Amer. Math. Soc.}, \textbf{33}(1) (1972), 8--14.

\bibitem{SA92}
M. K. Sen and M. R. Adhikari, On subtractive ideal of semirings, \textit{Internat. J.  Math. Sci.}, \textbf{15}(2) (1992), 347--350.

\bibitem{SA93}
\bysame, On maximal subtractive ideal of semirings, \textit{Proc. Amer. Math.
Soc.} \textbf{118} (1993), 699--703.
\end{thebibliography}
\end{document}